\documentclass[10pt]{amsart}
\usepackage[top=4.4cm,bottom=4.4cm,left=3.4cm,right=3.4cm]{geometry}               
\usepackage{amsmath,amscd,amssymb,amsthm}
\usepackage[english]{babel}
\usepackage{mathtools}
\usepackage{enumerate}
\usepackage{setspace}
\usepackage{mathrsfs,makecell}
\usepackage{color,array}
\usepackage[numbers]{natbib}
\usepackage[hidelinks]{hyperref}

\DeclareMathOperator{\Hom}{Hom}
\DeclareMathOperator{\Gal}{Gal}

\DeclareMathOperator{\Tr}{Tr}
\DeclareMathOperator{\N}{N}

\def\vF{\mathbb{F}}
\def\vN{\mathbb{N}}

\def\vP{\mathbb{P}}

\def\cO{\mathcal{O}}

\theoremstyle{definition}
\newtheorem{definition}{Definition}[section]
\newtheorem{remark}[definition]{Remark}
\newtheorem{problem}[definition]{Problem}

\theoremstyle{plain}
\newtheorem{theorem}[definition]{Theorem}
\newtheorem{corollary}[definition]{Corollary}
\newtheorem{lemma}[definition]{Lemma}
\newtheorem{proposition}[definition]{Proposition}

\newtheorem*{theorem*}{Theorem}

\newtheorem*{claim}{Claim}

\makeindex

\author[A. Ferraguti]{Andrea Ferraguti}
\address{Max Planck Institute for Mathematics\\
Vivatsgasse 7\\
53111 Bonn, Germany\\
}
\email{ferra@mpim-bonn.mpg.de}

\author[G. Micheli]{Giacomo Micheli}
\address{Mathematical Institute\\
University of Oxford\\
Woodstock Rd\\ 
Oxford OX2 6GG, United Kingdom
}
\email{giacomo.micheli@maths.ox.ac.uk}

\title[Classification of permutation rational functions of degree three]{Full Classification of permutation rational functions and complete rational functions of degree three over finite fields}

\thanks{The first author was supported by the Swiss National Science Foundation grant number 168459. The second author was supported by the Swiss National Science Foundation grant number 171248.}


\begin{document}

\begin{abstract}
Let $q$ be a prime power, $\vF_q$ be the finite field of order $q$ and $\vF_q(x)$ be the field of rational functions over $\vF_q$. In this paper we classify and count all rational functions $\varphi\in \vF_q(x)$ of degree 3 that induce a permutation of $\vP^1(\vF_q)$. As a consequence of our classification, we can show that there is no complete permutation rational function of degree $3$ unless $3\mid q$ and $\varphi$ is a polynomial.
\end{abstract}

\maketitle
\section{Introduction}

Let $q$ be a prime power and $\vF_q$ be the finite field of order $q$.
Writing permutation polynomial maps of $\vF_q$ (and their generalizations)
 is a matter of great interest in number theory and applied areas (see for example \cite{charpin2009does,charpin2008class,gao1997tests,konyagin2002enumerating, konyagin2006enumerating, laigle2007permutation,masuda2006number,rivest2001permutation,shparlinski1992deterministic,sun2005interleavers,xu2018constructions}).
The classification of permutation polynomials of degree $3$ has been longly known (see for example \cite[Table 7.1]{lidl1997finite}); the vast literature on the topic contains sparse results that, put together, can be used to describe exceptional functions of certain degrees, for example in terms of R\'edei functions. However, there is a lack of references that deal, in a compact and self-contained way, with the finite field theoretic framework. The goal of this paper is to provide, following existing ideas of Cohen and Reid for polynomials \cite{cohen1999permutation}, a systematic approach that can be used, in principle, to furnish an explicit description of permutation rational functions of any fixed degree over a finite field of arbitrary characteristic: we do not want to only classify pairs of geometric and arithmetic monodromy groups but we want to be able to list \emph{all} permutation rational functions of a given degree (and in turn give closed formulas for the number of permutation rational functions of given degree, as in Table \ref{tnumber_of_permutations}). We then show how to use this argument in order to list all rational functions of degree $3$, up to a suitable equivalence relation. This strategy has two immediate advantages: first, it yields a simple way to give an exact counting formula for the number of such functions; second, it allows to immediately classify complete permutation rational functions of degree three.

For applications, it is worth noticining that given a rational function $\varphi$ permuting $\vP^1(\vF_q)$, it is easy to construct a permutation of $\vF_q$. For example, one can use the method described in \cite{amadio2018full,guidi2018fractional}: if $a\in \vF_q$ is such that $\varphi(a)=\infty$ and $\varphi(\infty)=b$, the fractional jump construction provides a new function $\overline{\varphi}$ such that $\overline{\varphi}(x)=\varphi(x)$ for any $x\neq a$, and $\overline{\varphi}(a)=b$.
 
The methods we use to list permutation rational functions are mostly number theoretical: the strategy is to characterize such rational functions in terms of Galois group properties of function field extensions associated to them (similarily to \cite{guralnick2007exceptional}); from these properties we deduce equations which we are able to solve and whose solutions parametrise exactly permutation rational functions of degree 3. As a corollary of our results we also obtain a classification of complete permutation rational functions.

If $\mathcal M\subseteq \vF_q(x)$ is the set of M\"obius transformations of $\vP^1$, we say that two rational functions $\varphi,\psi\in \vF_q(x)$ are \emph{equivalent} and we write $\varphi\sim\psi$ if there exist $m_1,m_2\in \mathcal M$ such that $m_1\circ\varphi\circ m_2=\psi$. It is easy to see that $\sim$ is an equivalence relation on the set of permutation rational functions of fixed degree.
Summarising the results of Sections \ref{sec:odd_char} and \ref{sec:even_char}
we obtain the classification in Table \ref{table_permutation_up_to_equivalence}, which takes into account the aforementioned equivalence. 

\renewcommand{\arraystretch}{3}
\begin{center}
\begin{table}
\begin{tabular}{ | c | c | c | c | }
\hline
	$q$ & \makecell{Nr. of\\ eq. classes} & Representatives & Equations\\
	\hline
	$q\equiv 1 \bmod 6$  & 1 &  $\displaystyle \left[\frac{x^3+ax}{bx^2+1}\right]$     & $ab=9$, $-b$ not a square\\
   \hline
  $q\equiv 2 \bmod 6 $ & 1 & $[x^3]$ & ---\\
  \hline 
    $q\equiv 3 \bmod 6 $ & 2 & $[x^3]$, $[x^3+ax]$ &  $-a$ not a square \\
  \hline 

  $q\equiv 4 \bmod 6$ & 1 & $\displaystyle \left[\frac{x^3+a_2x^2+a_1x}{x^2+x+b_0}\right]$ & \makecell{$\displaystyle \Tr_{\vF_q/\vF_2}(b_0)=1$,\\ $\displaystyle a_1=b_0+\frac{1}{b_0}$, $\displaystyle a_2=1+\frac{1}{b_0}$}\\
  \hline
$q\equiv 5 \bmod 6 $ & 1 & $[x^3]$ & ---\\
   \hline
\end{tabular}
\vspace{0.5cm}
\caption{We denoted by $[\varphi]$ the equivalence class of a rational function $\varphi \in \vF_q(x)$ under the relation $\sim$ induced by composing on the left and on the right with M\"obius transformations. The table above lists all the permutation rational functions of degree $3$ up to the equivalence $\sim$.} 
\label{table_permutation_up_to_equivalence}
\end{table}

\end{center}

\begin{center}
\begin{table}
\begin{tabular}{| c | c |}
\hline
$q$ & \makecell{Nr. of\\ permutation rational functions \\ of degree 3} \\
\hline
$q\equiv 0 \mod 3$ & $\frac{1}{2}(q^4+q^3+q^2+q)$\\
\hline
$q\equiv 1 \mod 3$ &  $\frac{1}{2}(q^4-q^2)$ \\
\hline
$q\equiv 2 \mod 3$ & $\frac{1}{2}(q^2+q)^2$\\
\hline
\end{tabular}
\vspace{0.5cm}
\caption{The table above shows an exact formula for the number of permutation rational functions of degree $3$ depending on the congruence class of $q$ modulo $3$.} \label{tnumber_of_permutations}
\end{table}

\end{center}

Table \ref{table_permutation_up_to_equivalence} shows that in the cases in which degree 3 permutation polynomials exist (i.e.\ $q\equiv 0,2\bmod 3$), every permutation rational function of degree 3 is equivalent to a permutation polynomial. On the other hand, permutation rational functions of degree 3 exist for every $q$. Moreover, since the equivalence relation we use preserves separability, Table \ref{table_permutation_up_to_equivalence} also shows that for every $q$ there exists essentially one separable permutation rational function of degree 3 (up to equivalence).

As a consequence of these results, notice that all permutation rational functions of degree 3  are exceptional (i.e. they permute infinitely many extension fields), as for any choice of coefficients verifying the conditions over a certain field $\vF_q$, there exist infinitely many extensions $\vF_{q^n}$ of $\vF_q$ such that the same conditions are verified.

For example, if $q$ is a prime congruent to $ 7 \bmod 12$ (i.e. $q\equiv 1 \bmod 6$ and $-1$ is not a square in $\vF_p$), all permutation rational functions of degree 3 are equivalent to $\varphi={\frac{x^3+x}{9x^2+1}}$. In addition, if $e$ is the order of $q$ modulo $12$ then $\varphi$ permutes (for example) $\vF_{p^{(e+1)^d}}$ for \emph{any} positive integer $d$.

The methods we explain in the paper easily allow to give an exact counting of the permutation rational functions of degree $3$ over a finite field which just depends on the congruence class of $q$ modulo $3$. The formulas are shown in Table
\ref{tnumber_of_permutations} and they address \cite[Problem 8.1.13]{mullen2013handbook} and \cite[Problem 14]{mullen1996open} for the case $n=3$ but in the more general framework of rational functions.

We now give a brief outline of the paper. Section \ref{sec:preliminaries} is devoted to a recap of the two main tools used, which are Chebotarev Density Theorem for function fields (Theorem \ref{chebotarev_density_theorem}) and a result connecting ramification and splitting of places in global function field extensions and their Galois closures (Lemma \ref{orbits}). In Section \ref{sec:perm_and_fixed_points} we assemble the preliminaries to produce, following ideas of Cohen \cite{cohen1999permutation}, a handy criterion (Theorem \ref{fixed_points}) to establish whether a rational function induces a bijection of $\vP^1(\vF_q)$. In Section \ref{sec:degree_three} the criterion is then applied to characterize rational functions of degree $3$ which are permutations  (Lemma \ref{discriminant}).
In Sections \ref{sec:odd_char} and \ref{sec:even_char} we use (3) and (4) of Lemma \ref{discriminant} to find equations for the coefficients of the permutation rational functions of degree $3$ (Theorem \ref{complete_classification}, Corollary \ref{cor:equivalence_odd_char}, Theorem \ref{char_2_class}, and Corollary \ref{cor:equivalence_char_2}).
Finally, in Section \ref{sec:ComRatDeg3} we apply our results to show that there is no complete permutation rational function of degree $3$ unless $3\mid q$ and $\phi$ is a polynomial.

\section{Preliminary results}\label{sec:preliminaries}
In what follows we will use notation and terminology of \cite{stichtenoth2009algebraic}. Let us recall a classical result from algebraic number theory.
\begin{lemma}\label{orbits}
Let $L:K$ be a finite separable extension of global function fields and let $M$ be its Galois closure with Galois group $G$. Let $P$ be a place of $K$ and $\mathcal Q$ be the set of places of $L$ lying above $P$.
Let $R$ be a place of $M$ lying above $P$. The following hold:
\begin{enumerate}
\item There is a natural bijection between $\mathcal Q$ and the set of orbits of $H\coloneqq\Hom_K(L,M)$ under the action of the decomposition group $D(R|P)=\{g\in G\,|\, g(R)=R\}$.
\item  Let $Q\in \mathcal Q$ and let $H_Q$ be the orbit of $D(R|P)$ corresponding to $Q$. Then $|H_Q|=e(Q|P)f(Q|P)$ where $e(Q|P)$ and $f(Q|P)$ are ramification index and relative degree, respectively. 

\item The orbit $H_Q$ partitions further under the action of the inertia group $I(R|P)$ into $f(Q|P)$ orbits of size $e(Q|P)$. 

\end{enumerate}
\end{lemma}
A proof of Lemma \ref{orbits} can be found for example in \cite{guralnick2007exceptional}.

Let now $F/\vF_q$ be a global function field with full constant field $\vF_q$, and let $M/F$ be a finite Galois extension.
Let $k=\overline{\vF}_q\cap M$ be the field of constants of $M$.
Notice that $\Gal(kF:F)\cong \Gal(k:\mathbb F_q)\cong C_{[k:\vF_q]}$.
Let $\gamma\in G\coloneqq\Gal(M:F)$ be such that $\phi\coloneqq\gamma_{|k}$ is the Frobenius automorphism of $k:\vF_q$ (i.e.\ the map $a\mapsto a^q$).
Let $R\subseteq M$ be a place of $M$ lying above a place $P$ of degree $1$. Let $\cO_R$ and $\cO_P$ be the valuation rings of $P$ and $R$, respectively.
Elements of $D(R|P)$ naturally restrict to well-defined automorphisms of $\cO_R/R$, giving a surjective map
\[\pi_R: D(R|P)\twoheadrightarrow \Gal(\cO_R/R:\cO_P/P).\]
We set
\[D_{\phi}(R|P)\coloneqq\pi_R^{-1}(\phi)\]
and let $N=\Gal(M:kF)$. For any element $\sigma\in N\gamma$ let us define 
\[w_P(\sigma)\coloneqq\frac{\# (D_\phi(R|P)\cap \Gamma_\sigma)}{\#\Gamma_\sigma \cdot \#D_\phi(R|P)},\]
where $\Gamma_\sigma$ is the conjugacy class of $\sigma$ in $G$. Notice that since $G/N$ is cyclic we have that $\Gamma_\sigma\subseteq N\gamma$.
\begin{theorem}[Chebotarev]\label{chebotarev_density_theorem}
Let $M,F,k,G,N$ be as above. Let $g_M$ be the genus of $M$.
Let $\sigma\in N\gamma$. Then
\[\left| \sum_{P\in \mathcal P^1(F/\vF_q)} w_P(\sigma)- \frac{1}{\# N} (q+1)\right|\leq \frac{2}{\# N}g_M \sqrt q,\]
where $\mathcal P^1(F/\vF_q)$ is the set of places of degree $1$ of $F$.
\end{theorem}
The proof of this inequality comes essentially from an application of the Riemann Hypothesis for curves over a finite field to a certain curve associated to $\sigma$. See \cite{kosters2017short} for details.
\begin{remark}
Let $\sigma\in N\gamma$.
Theorem \ref{chebotarev_density_theorem} can be restated as follows: \[\left| \sum_{\text{ramified }  P\in \mathcal P^1(F/\vF_q)}w_P(\sigma)\quad + 
\sum_{\substack{\text{unramified } P\in \mathcal P^1(F/\vF_q): \\ \text{ $\sigma$ is a Frobenius at $P$}}} \frac{1}{\#\Gamma_\sigma}- \frac{1}{\# N} (q+1)\right|\leq \frac{2}{\# N}g_M \sqrt q.\]

By multiplying by $\#\Gamma_\sigma$ and ignoring the ramified places (as $w_P(\sigma)\in [0,1]$) one gets the usual estimate for the number of $P\in \mathcal P^1(F/\vF_q)$ such that $\sigma$ is a Frobenius at $P$, which in $q$ is asymptotic to $\displaystyle \frac{\# \Gamma_\sigma}{\# N} (q+1)$.
\end{remark}

\section{Permutations and fixed points: a toolkit for constructing permutations over \texorpdfstring{$\vF_q$}{}}\label{sec:perm_and_fixed_points}
In this section we provide the main tool which we will use in the rest of the paper (see also \cite[Theorem 9.7.24]{mullen2013handbook} for an equivalent condition). These ideas  date back to Cohen and Fried for polynomials \cite{cohen1999permutation}; although the proof for rational functions follows the same lines, we report it here for the sake of completeness: the purpose of this section is to be self-contained and allow a non-expert to assimilate the method quickly from a compact source.

\begin{definition}
 A rational function $\varphi\in \vF_q(x)$ is called a \emph{permutation rational function} if it induces a permutation of $\vP^1(\vF_q)$.
\end{definition}

\begin{theorem}\label{fixed_points}
Let $p$ be a prime and $q=p^\ell$ for some positive integer $\ell$.
Let $f,g$ be coprime polynomials in $\vF_q[x]$ not both lying in $\vF_q[x^p]$ and let $\varphi\coloneqq f/g$. Let $M$ be the splitting field of $f-tg$ over $\vF_q(t)$, $k=M\cap \overline{\vF}_q$ be the field of constants of $M$, $G=\Gal(M:\vF_q(t))$ and $N=\Gal(M:k(t))$.  Let $\gamma\in G$ be a Frobenius for the field of constants extension (i.e.\ such that $\gamma_{|k}=a\mapsto a^q$). Let $\alpha$ be a root of $f-tg$ and let $H\coloneqq\Hom_{\vF_q(t)}(\vF_q(\alpha), M)$. Then there exists a constant C, depending only on the degrees of $f$ and $g$, such that if $q\geq C$, then $\varphi$ is permutation if and only if all elements in $N\gamma$ have exactly one fixed point when acting on $H$.
\end{theorem}
\begin{remark}
Notice that $H$ is in natural correspondence with the roots of $f(x)-t$ over $\overline{\vF_q(t)}$.
\end{remark}
\begin{proof}[Proof of Theorem \ref{fixed_points}]
Suppose first that $\varphi$ is permutation and let $\sigma\in N\gamma$. For convenience, let us identify the places of degree $1$ of $\vF_q(t)$ with $\vF_q\cup P_\infty$ where $P_\infty$ is the place at infinity of $\vF_q(t)$. 
Now consider the quantity
\begin{equation}\label{choice_of_C}
S\coloneqq q+1-2g_M\sqrt q \;- \sum_{\text{ramified }  P\in \mathcal P^1(\vF_q(t)/\vF_q)}w_P(\sigma),
\end{equation}
where $g_M$ and $w_P(\sigma)$ are defined as in Section \ref{sec:preliminaries}. Since the genus $g_M$ and the number of ramified places can be bounded by a constant which depends only on the degrees of $f$ and $g$ (and is therefore independent of $q$), there exists a constant $C$ such that if $q\geq C$, then $S\geq 2$. On the other hand, it follows From Theorem \ref{chebotarev_density_theorem} that if $S\geq 2$ then $\vF_q(t)$ has an unramified place of degree 1 different from $P_{\infty}$, which we can identify with some $t_0\in \vF_q$, for which $\sigma$ is a Frobenius. Let now $q\geq C$ and $D=\langle \sigma \rangle\subseteq G$. Since $t_0$ is unramified, $D$ is the decomposition group of a place of $M$ lying above $t_0$. The factorization pattern of $f-t_0g$ over $\vF_q[x]$ corresponds to the orbits of the action of $D$ on $H$ by Lemma \ref{orbits}. Since $\varphi$ is a permutation, $f-t_0g$ must have exactly one linear factor. Since $\sigma$ generates $D$ it follows that it has
  exactly one fixed point.

Conversely, suppose that all elements in $N\gamma$ have exactly one fixed point, and by contradiction assume that $\varphi$ is not a permutation of 
$\vP^1(\vF_q)$. Let $P$ be a place of degree one of $\vF_q(t)$ such that $P$ is outside the image of $\varphi$. Then, for any place $Q$ of $\vF_q(\alpha)$ lying above $P$ we have $f(Q|P)\geq 2$. Let now $R$ be a place of $M$ lying above $P$.
By Lemma \ref{orbits} we have that for any $Q$ lying above $P$, the orbit of $D(R|P)$ on $H$  has size greater or equal than $2e(Q|P)$. Let now $\sigma$ be a Frobenius for $P$ in $D(R|P)$. Then $\sigma\in N\gamma$, so it has a fixed point $z\in H$. Notice that in general this is  an orbit of size one of $\sigma$, not an orbit of size one of $D(R|P)$. Since we can write 
\[D(R|P)=\bigcup^{f(R|P)}_{i=1} I(R|P)\sigma^i,\]
where $I(R|P)$ is the inertia subgroup, we see that, by point (3) of Lemma \ref{orbits}, 
\[|D(R|P)\cdot z|=\left|\bigcup^{f(R|P)}_{i=1} I(R|P)\sigma^i\cdot z\right|=|I(R|P)\cdot z|=e(Q'|P),\]
where $Q'$ is the place of $\vF_q(\alpha)$ corresponding to the orbit of $z$ under the action of $D(R|P)$. A contradiction follows, since this orbit must have size greater than or equal than $2e(Q'|P)$.
\end{proof}

\begin{remark}\label{suff}
It is clear from the proof of the theorem that if all elements in $N\gamma$ have exactly one fixed point, then $\varphi$ is permutation, independently of the size of $q$. On the other hand, when $q<C$ it is possible to find examples of permutation rational functions such that not all elements in $N\gamma$ have exactly one fixed point when acting on $H$. It is worth noticing that these do not appear in degree 3, while they do in degree 4.
\end{remark}

\begin{corollary}\label{fixed_points_cor}
Let $\varphi=f/g\in \vF_q(x)$ be a separable permutation rational function and $q\geq C$. Then the field of constants $k$ of the splitting field of $f-tg$ properly contains $\vF_q$.
\end{corollary}
\begin{proof}
In the notation of Theorem \ref{fixed_points}, if $k=\vF_q$ then $\gamma=1$ so $N\gamma=G$. But the identity has more than one fixed point, and the claim follows.
\end{proof}

The combination of the above results together with the equivalence relation $\sim$ mentioned in the introduction will allow an exact counting of the permutation rational functions of degree $3$.

\section{Permutation rational functions of degree three}\label{sec:degree_three}
The goal of this section is to convert the group theoretical condition obtained in Theorem \ref{fixed_points} into equations for the coefficients of permutation rational functions of degree 3.
We first need a technical lemma which allows  us to control the constant $C$ appearing in Theorem \ref{fixed_points}.
\begin{lemma}\label{lemma_control_C}
Let $\varphi=f/g\in \vF_q(x)$ be a separable rational function of degree 3. Then the constant $C$ appearing in Theorem \ref{fixed_points} can be chosen less than $13$.
\end{lemma}
\begin{proof}
The constant $C$ in Theorem \ref{fixed_points} comes from the quantity \eqref{choice_of_C}, which we want to be greater or equal than $2$. In order to obtain this, we need explicit upper bounds for the genus $g_M$ and the number of ramified places of the splitting field $M$ of $f-tg$. 
The genus $g_M$ of $M$ can be easily bounded by Riemann's Inequality as follows. For any two roots $\alpha,\beta$ of $f-tg$, seen as a polynomial over $\vF_q(t)$, we have that $M=\vF_q(\alpha,\beta)$ (notice that $t$ can be obtained by computing $f(\alpha)/g(\alpha)$). Now observe that $[M:\vF_q(\alpha)]=[M:\vF_q(\beta)]\leq 2$ and apply \cite[Corollary 3.11.4]{stichtenoth2009algebraic} with $F=M$ and $K=\vF_q$ to see immediately that $g_M\leq 1$.

To bound the term $\sum_{\text{ramified }  P\in \mathcal P^1(\vF_q(t)/\vF_q)}w_P(\sigma)$ of \eqref{choice_of_C}, notice that it is enough to bound the number of places of degree 1 of $\vF_q(t)$ that ramify in $M$, as $0\leq w_P(\sigma)\leq 1$. Notice that all places of $\vF_q(t)$ ramifying in $M$ already ramify in $\vF_q(\alpha)$, because $M$ is the Galois closure of $\vF_q(\alpha):\vF_q(t)$. Now apply Hurwitz genus formula (see \cite[Theorem 3.4.13]{stichtenoth2009algebraic}) to see that there are at most $4$ places of $\vF_q(t)$ that ramify in $\vF_q(\alpha)$. 

All in all, it follows that if $q\geq 13$, then
\[S=q+1-2g_M\sqrt q - \sum_{\text{ramified }  P\in \mathcal P^1(F/\vF_q)}w_P(\sigma)\geq q+1-2\sqrt q - 4\geq  2\]
\end{proof}

We are now ready to start the classification of permutation rational functions of degree 3. Let us first recall the following definition.

\begin{definition}
 Let $K$ be a field and $f=x^3+ax^2+bx+c\in K[x]$. The \emph{quadratic resolvent} of $f$ is the polynomial defined by:
 $$R_2(x)\coloneqq x^2+(ab - 3c)x + (a^3c + b^3 + 9c^2 - 6abc).$$
 Equivalently, if $r_1,r_2,r_3$ are the roots of $f$ in $\overline{K}$, the quadratic resolvent is given by
 $$R_2(x)=(x-(r_1^2 r_2 + r_2^2 r_3 + r_3^2 r_1))(x-(r_2^2r_1 + r_1^2r_3 + r_3^2r_2)).$$
\end{definition}

Let now $\varphi=f/g\in \vF_q(x)$ be a separable rational function of degree 3, where $f,g$ are coprime polynomials. Let $M$ be the splitting field of $f-tg$ over $\vF_q(t)$, and let $k$ be its field of constants. Let $G\coloneqq\Gal(M/\vF_q(t))$ and $N\coloneqq \Gal(M/k(t))$. Note that $N\trianglelefteq G$.

\begin{lemma}\label{discriminant}
 Let $q\geq 13$. Let $R_2(x)\in \vF_q(t)[x]$ be the quadratic resolvent of the polynomial $f-tg\in \vF_q(t)[x]$. The following three conditions are equivalent.
 \begin{enumerate}
  \item The function $\varphi=f/g$ is permutation.
  \item The arithmetic Galois group $G$ is $S_3$ and the geometric Galois group $N$ is $A_3$.
  \item The quadratic resolvent of $f-tg$ is irreducible in $\vF_q(t)[x]$ but reducible in $\vF_{q^2}(t)[x]$.
  \end{enumerate}
  If in addition $q$ is odd, the conditions above are equivalent to the following.
  \begin{enumerate}
  \item[(4)] Let $\Delta(t)\in \vF_q[t]$ be the discriminant of $f-tg$. There exists $u\in \vF_q\setminus\vF_q^2$ such that $\Delta(t)=u\cdot r(t)^2$, where $r(t)\in \vF_q[t]$.
 \end{enumerate}
\end{lemma}
\begin{proof}
 First, suppose that $\varphi$ is a permutation rational function. Since $f-tg$ is irreducible, $G$ is a transitive subgroup of $S_3$ and thus there are only two possibilities: $G\simeq A_3$ or $S_3$. If $G\simeq A_3$, then $N=\{1\}$ because of Corollary  \ref{fixed_points_cor}. But this means $f-tg$ splits over $\overline\vF_q(t)$, which is a contradiction, as it is geometrically irreducible. 
 Thus, $G\simeq S_3$ and for the same reasons one must have $N\simeq A_3$, since the field of constants $k$ is non-trivial, which forces also $[k:\vF_q]=2$.
 
 Conversely, the non-trivial coset of $A_3$ in $S_3$ contains the transpositions $(12),(13),(23)$, all of which have exactly one fixed point. Thus, in this case $\varphi$ is a permutation rational function by Theorem \ref{fixed_points}. This proves that (1) and (2) are equivalent.
 
 To prove that (2) and (3) are equivalent, first notice that if $R_2(x)$ is irreducible over $\vF_q(t)[x]$, then $F\coloneqq\vF_q(t)[x]/(R_2(x))$ is a subfield of $M$. Moreover, the Galois group of a separable irreducible cubic is $S_3$ if and only if its quadratic resolvent is irreducible.  
 Thus, if (2) holds then $F$ is the unique quadratic extension of $\vF_q(t)$ contained in $M$, and it must coincide with the fixed field of $N$, so\ $F=\vF_{q^2}(t)$. This implies immediately (3). Conversely, if (3) holds then $G\simeq S_3$ and the fixed field of $N$ is $\vF_{q^2}(t)$, easily implying (2).
 
 To finish the proof, notice that if $q$ is odd, (3) is equivalent to the discriminant of $R_2(x)$ being a square in $\vF_{q^2}(t)$ but not in $\vF_q(t)$. Since the discriminant of a cubic polynomial coincides with the discriminant of its quadratic resolvent, (3) and (4) are clearly equivalent.
\end{proof}

\begin{remark}\label{sufficient_condition}
 The proof of the above lemma, together with Remark \ref{suff}, shows that the only implication that needs the hypothesis $q\geq 13$ is (1)$\Rightarrow$(2). Every other implication holds for every $q$.
\end{remark}

As a corollary of the above lemma, we can immediately recover the classification of normalized permutation polynomials of degree 3 for $q\geq 13$, without using Hermite's criterion. Recall that a polynomial of degree $3$ is called \emph{normalized} if it is of the form $x^3+ax$ when $3\nmid q$, or of the form $x^3+ax^2+bx$ when $3\mid q$. The reason for this is that any polynomial of degree 3 can be brought in one of these forms via a transformation of the shape $f\mapsto uf(vx+w)+z$ for some $u,v,w,z\in \vF_q$. 

\begin{corollary}\label{degree_3_class}
 Let $q\geq 13$ and let $f\in \vF_q[x]$ be a normalized polynomial. Then $f$ is permutation if and only if either $q\equiv 2 \bmod 3$ and $f=x^3$ or $q\equiv 0 \bmod 3$ and $f=x^3+bx$ where $b=0$ or $-b$ is not a square. 
\end{corollary}
\begin{proof}
 First, let $3\mid q$ and $f=x^3+ax^2+bx$ for some $a,b\in \vF_q$. If $a=b=0$, then $f$ is permutation. Otherwise, $f$ is separable and we can apply Lemma \ref{discriminant}: the discriminant of $f-t$ is given by $4a^2t-4b^3+a^2b^2$. Condition (4) then holds if and only if $a=0$ and $-b$ is not a square.
 
 If $3\nmid q$, let $f=x^3+bx$. The quadratic resolvent of $f-t$ is given by $R_2=x^2+3tx+b^3+9t^2$. By Lemma \ref{discriminant}, $f$ is permutations if and only if $R_2$ is irreducible over $\vF_q(t)$ but not over $\vF_{q^2}(t)$. This is equivalent to ask the existence of $u,v\in \vF_{q^2}$, not both in $\vF_q$, such that $R_2=(x-(ut+v))(x-(u^qt+v^q))$. One checks immediately that this condition implies that $\N_{\vF_{q^2}/\vF_q}(u)=9$ and $v(3+2u)=0$. The solution $3+2u=0$ is incompatible with the norm condition, so we must have $v=0$. Thus, $b=0$ and of course $f=x^3$ is permutation if and only if $q\equiv 2 \bmod 3$.
\end{proof}

For every $n\in\vN$, let now $R_n\subseteq \vF_q(x)$ be the set of rational functions of degree $n$. Let $\displaystyle \mathcal M=\left\{\frac{ax+b}{cx+d}\colon ad-bc\neq 0\right\}\subseteq \vF_q(x)$ be the set of M\"obius transformations. $\mathcal M$ acts on $R_n$ on the left and on the right by composition. Clearly, this action restricts to an action on the set of permutation rational functions of degree $n$.
\begin{definition}\label{def:eqrelation}
 We say that two rational functions $\varphi,\psi\in R_n$ are \emph{equivalent} if they lie in the same double coset for the action of $\mathcal M$ on $R_n$, i.e.\ if there exist $m_1,m_2\in \mathcal M$ such that $m_1\circ\varphi\circ m_2=\psi$.
\end{definition}

\section{The odd characteristic case}\label{sec:odd_char}
Throughout this section, we assume that $q$ is odd.
\begin{theorem}\label{complete_classification}
 Let $\varphi\in \vF_q(x)$ be a separable permutation rational function of degree 3. Then the following hold.
 \begin{enumerate}
  \item $\varphi$ is equivalent to a rational function of the form $\displaystyle \frac{x^3+ax}{bx^2+1}\in \vF_q(x)$, where $a,b\in\vF_q$ and $b\neq 0$.
  \item Let $a,b\in \vF_q$ and $b\neq 0$. A rational function $\displaystyle \frac{x^3+ax}{bx^2+1}\in \vF_q(x)$ is permutation if and only if $-b$ is not a square and $ab=9$.
 \end{enumerate}
\end{theorem}
\begin{proof}
Let $f,g\in \vF_q[x]$ be coprime polynomials such that $\varphi=f/g$.

First, we claim that $\varphi$ is equivalent to a rational function of the form $f'/g'$ where $f',g'\in\vF_q[x]$ are such that $\deg f'=3$ and $\deg g'=2$.

Since $\deg \varphi=3$, at least one between $f$ and $g$ has degree 3, so up to composing $\varphi$ on the left with $1/x$ we can assume that $\deg f=3$. If $\deg g=2$, there is nothing to prove. If $\deg g=1$, then $\varphi(\infty)=\infty$ but $g$ has a root $\alpha\in \vF_q$ and thus $\varphi(\alpha)=\infty$, which is a contradiction because $\varphi$ is permutation. If $\deg g=0$, $\varphi$ is a polynomial of degree 3 and we can assume that $\varphi=x^3+ax^2+bx$ for some $a,b\in\vF_q$. If $b=0$ then $a=0$ because $\varphi$ cannot have more than one root. In this case, since $\varphi$ is separable, we must have $3\nmid q$. Then 
\[ \frac{x}{x-1}\circ x^3\circ \frac{x}{x+1}=\frac{x^3}{-3x^2-3x-1}.\] Otherwise, $b\neq 0$ and \[ \frac{1}{x}\circ \varphi\circ \frac{1}{x}=\frac{x^3}{bx^2+ax+1}.\] Finally, if $\deg g=3$ and the coefficient of its leading term is $b_3$, composing $\varphi$ on the left with $x/(x-1/b_3)$ yields a function $\varphi'=f/g'$ where $\deg g'\leq 2$. This proves the existence of $f',g'$ as in the claim.

Now suppose $g'(x)=c_2x^2+c_1x+c_0$ for some $c_0,c_1,c_2\in \vF_q$ with $c_2\neq 0$. Composing $f'/g'$ on the right by $x-c_1/(2c_2)$ yields a function 
\[\varphi'=\displaystyle \frac{x^3+a_2'x^2+a_1'x+a_0'}{c_2x^2+c_0'}.\] 
Notice that $c_0'\neq 0$ because the denominator cannot have roots. Composing $\varphi'$ on the left by $x-a_0'/c_0'$ and multiplying the whole function by $c_0'$ gives us a function of the form $\displaystyle \frac{x^3+a_2x^2+a_1x}{b_2x^2+1}$ for some $a_1,a_2,b_2\in\vF_q$ with $b_2\neq 0$.

Next, we are going to show that an element $\displaystyle \psi=\frac{x^3+a_2x^2+a_1x}{b_2x^2+1}$ with $b_2\neq 0$ is permutation if and only if $a_2=0$, $-b_2$ is not a square and $a_1b_2=9$. This will conclude the proof of the theorem.

To prove it, we first let $q\geq 13$ and apply Lemma \ref{discriminant}. Notice that the hypothesis of the lemma hold since the fact that $b_2\neq 0$ implies that at least one between $x^3+a_2x^2+a_1x$ and $b_2x^2+1$ is separable. The discriminant of $x^3+a_2x^2+a_1x-t(b_2x^2+1)$ is given by:
$$-4b_2^3\left(\underbrace{t^4 - \frac{3a_2}{b_2}t^3 - \frac{a_1^2b_2^2 + 18a_1b_2 - 12a_2^2b_2 - 27}{4b_2^3}t^2 + \frac{a_1^2a_2b_2 + 9a_1a_2 - 2a_2^3}{2b_2^3}t +\frac{ 4a_1^3 - a_1^2a_2^2}{4b_2^3}}_{h(t)}\right).$$
By condition (4) of Lemma \ref{discriminant}, $\psi$ is permutation if and only if the polynomial $h(t)$ inside the brackets is a square in $\vF_q[t]$ and $-b_2$ is not be a square in $\vF_q$.

We claim that if $\psi$ is permutation, then $a_2=0$. To show this, first notice that if $t^4+At^3+Bt^2+Ct+D\in \vF_q[t]$ is a square then:
$$\begin{cases}
   A^3+8C-4AB=0 & \\
   DA^2-C^2=0 & \\
  \end{cases}.
$$
Now substitute the coefficients of $h(t)$ into these equations and assume by contradiction that $a_2\neq 0$. This gives the following system:
  $$\begin{cases}
   -a_1^2b_2^2-a_2^2b_2+18a_1b_2-81=0 & \\
   -a_1^4b_2^2-5a_1^2a_2^2b_2-4a_2^4+18a_1^3b_2+36a_1a_2^2-81a_1^2=0 & .
  \end{cases}
  $$
  The first equation yields: \[ a_2^2=\frac{18a_1b_2-a_1^2b_2^2-81}{b_2}=\frac{-(a_1b_2-9)^2}{b_2},\] and substituting into the second one it follows that $36(a_1b_2-9)^3=0$. If $3\nmid q$, then $a_1b_2=9$ and consequently $a_2=0$, which is a contradiction. On the other hand, if $3\mid q$ one gets from the first equation $a_2^2=-a_1^2b_2$. If $a_1\neq 0$ then $-b_2$ is a square, which is a contradiction. Thus, $a_1=0$ and so $a_2=0$.
  
  Therefore, 
  \[ \frac{x^3+a_2x^2+a_1x}{b_2x^2+1}\quad \text{is permutation}\Longleftrightarrow \begin{cases} a_2=0, -b_2 \;\text{is not a square in $\vF_q$}\\ t^4-\frac{a_1^2b_2^2+18a_1b_2-27}{4b_2^3}t^2+\frac{a_1^3}{b_2^3}=(t^2+u)^2\end{cases}\]
  for some $u\in \vF_q$. 
  Of course this holds if and only if $a_2=0$, $-b_2$ is not a square and
 $$\begin{dcases}
    \frac{a_1^3}{b_2^3}=u^2 & \\
    -\frac{a_1^2b_2^2+18a_1b_2-27}{4b_2^3}=2u & 
   \end{dcases},$$
  which in turn is equivalent to $(a_1b_2-1)(a_1b_2-9)^3=0$. But $a_1b_2=1$ cannot hold, because in that case we would have $\varphi=a_1x$. Therefore we must have $a_1b_2=9$, and the claim follows.
  
  When $q< 13$, Remark \ref{sufficient_condition} shows that the conditions $-b_2\notin\vF_q^2$ and $a_1b_2=9$ are still sufficient for $\displaystyle \frac{x^3+a_1x}{b_2x^2+1}$ to be permutation. A complete check of the remaining cases with Magma \cite{MR1484478} shows that they are also necessary.
\end{proof}

\begin{remark}
 Since a non-square in $\vF_q$ remains a non-square in $\vF_{q^n}$ if and only if $n$ is odd, Theorem \ref{complete_classification} shows that a separable permutation rational function of degree 3 in $\vF_q(x)$ permutes $\vP^1(\vF_{q^n})$ if and only if $n$ is odd.
\end{remark}

\begin{remark}
Notice that (2) of Theorem \ref{complete_classification} can also be deduced using the theory of Redei functions.
\end{remark}

The following corollary shows that for odd $q$, any two separable rational functions of degree 3 in $\vF_q(x)$ are equivalent. 

\begin{corollary}\label{cor:equivalence_odd_char}
 The following hold.
 \begin{enumerate}
  \item Let $\varphi,\psi\in \vF_q(x)$ be two separable permutation rational functions of degree 3. Then $\varphi\sim\psi$.
  \item An inseparable rational function $\varphi\in\vF_q(x)$ of degree 3 is permutation if and only if $3\mid q$ and $\varphi\sim x^3$. 
 \end{enumerate}
\end{corollary}
\begin{proof}
 (1) By Theorem \ref{complete_classification}, we can assume that 
 \[ \varphi=\frac{x^3+ax}{bx^2+1}\quad \text{and} \quad \psi=\frac{x^3+a'x}{b'x^2+1}\] for some $a,a',b,b'\in \vF_q$ with $bb'\neq 0$, $-b,-b'\notin\vF_q^2$ and $ab=a'b'=9$. Thus there exists $k\in \vF_q^*$ such that $b=b'k^2$. It follows immediately that \[\frac{x}{k^3}\circ\psi\circ kx=\varphi.\]
 
 (2) Let $\varphi=f/g$. Since $\deg\varphi=3$, $\varphi$ is inseparable if and only if $3\mid q$ and $f,g$ are both inseparable polynomials of degree 3. This means that $f=x^3+a$ and $g=x^3+b$ for some $a,b\in \vF_q$. But then 
 \[\varphi=\left(\frac{x+a}{x+b}\right)^3,\] 
 which is of course equivalent to $x^3$. Conversely, if $3\mid q$ and $\varphi$ is equivalent to $x^3$, then clearly $\varphi$ is permutation.
\end{proof}

\begin{corollary}\label{number_of_permutations}
 Let $N_q$ be the number of rational functions of degree 3 of the form $f/g$, where $f,g\in \vF_q[x]$ are monic coprime polynomials.
 \begin{enumerate}
  \item If $3\nmid q$, let $m\in\{1,2\}$ be the residue class of $q$ modulo 3. Then:
  $$N_q=\frac{1}{2}(q^4+2(m-1)q^3+(2m-3)q^2).$$
  \item If $3\mid q$ then:
  $$N_q=\frac{1}{2}(q^4+q^3+q^2+q).$$
 \end{enumerate}
 \begin{proof}
We will prove (1) and (2) at the same time. For $s,t\in\vN$, let $R_{s,t}$ be the set of permutation rational functions in $\vF_q(x)$ of the form $f/g$, where $f,g\in \vF_q[x]$ are monic, coprime polynomials such that $\deg f=s$ and $\deg g=t$. Since $\varphi$ is permutation if and only if $1/\varphi$ is permutation, we see immediately that $N_q=|R_{3,3}|+2|R_{3,2}|+2|R_{3,0}|$ (notice that $R_{3,1}=\emptyset$).
  
  Let us start by computing $|R_{3,0}|$. This is just the number of monic permutation polynomials, and it is easy to check that it is $(m-1)q^2$ if $3\nmid q$ and $(q^2+q)/2$ otherwise.
  
  In order to compute $|R_{3,2}|$, we first compute $|R_{3,2}'|$, where $R_{3,2}'\subseteq R_{3,2}$ is the set of elements of the form $\displaystyle \frac{x^3+a_2x^2+a_1x+a_0}{x^2+b_0}$. Composing on the left with $x-a_0/b_0$ and using Theorem \ref{complete_classification}, one checks that such elements are permutation if and only if $-b_0\notin \vF_q^2$, $a_1/b_0=9$ and $a_2=a_0/b_0$. This shows that the choices of $a_0,b_0$ determine the function completely. Therefore, $|R_{3,2}'|=q(q-1)/2$. Now consider the following map:
  $$\pi\colon R_{3,2}\rightarrow R_{3,2}'$$
  $$\varphi=\frac{x^3+a_2x^2+a_1x+a_0}{x^2+b_1x+b_0}\mapsto \varphi\circ \left(x-\frac{b_1}{2}\right).$$
  One checks that $\pi$ is surjective and its fibers have cardinality $q$. This proves that $|R_{3,2}|=q^2(q-1)/2$.
  
  The remaining step is to compute $|R_{3,3}|$. Let first $R_{3,3}'\subseteq R_{3,3}$ be the subset of functions of the form 
  \[\varphi=\frac{x^3+a_2x^2+a_1x+a_0}{x^3+b_2x^2+b_1x+b_0}\quad \text{with} \quad b_2\neq a_2.\] Again, one defines a function
  $$\pi\colon R_{3,3}'\rightarrow R_{3,2}$$
  $$\varphi\mapsto \frac{(a_2-b_2)x}{x-1}\circ\varphi$$
  and checks that this is surjective with fibers of cardinality $q-1$. It follows that $|R_{3,3}'|=q^2(q-1)^2/2$.
  
  Finally, we need to compute $|R_{3,3}\setminus R_{3,3}'|$. Let 
  \[ \varphi=\frac{x^3+a_2x^2+a_1x+a_0}{x^3+a_2x^2+b_1x+b_0}\in R_{3,3}\setminus R_{3,3}'.\] The function \[ \frac{x}{x-1}\circ \varphi=\frac{x^3+a_2x^2+a_1x+a_0}{(a_1-b_1)x+a_0-b_0}\] is a permutation rational function of degree 3. This implies that $b_1=a_1$, $a_0\neq b_0$ and $x^3+a_2x^2+a_1x+a_0\in R_{3,0}$. Conversely, for every choice of $x^3+a_2x^2+a_1x+a_0\in R_{3,0}$ there are exactly $(q-1)$ choices of $b_0$ such that $\displaystyle \frac{x^3+a_2x^2+a_1x+a_0}{x^3+a_2x^2+a_1x+b_0}$ is permutation. It follows that $R_{3,3}\setminus R_{3,3}'=(q-1)|R_{3,0}|$.
  
 \end{proof}

\end{corollary}

\section{The even characteristic case}\label{sec:even_char}

Throughout this section, we let $q=2^n$ for some fixed $n\in \vN$. Let us start by recalling the following elementary lemma.
\begin{lemma}\label{quadratic_eq}
 A quadratic polynomial $x^2+bx+c\in \vF_q[x]$ with $b\neq 0$ has a root in $\vF_q$ if and only if $\displaystyle\Tr_{\vF_q/\vF_2}\left(\frac{c}{b^2}\right)=0$.
\end{lemma}
\begin{proof}
 See \cite[Remark 11.1.120]{mullen2013handbook}.
\end{proof}

\begin{theorem}\label{char_2_class}
 Let $\varphi\in \vF_q(x)$ be a rational function of degree 3. The following hold:
 \begin{enumerate}
  \item if $\varphi$ is permutation, then it is equivalent to a function of the form $\displaystyle \frac{x^3+a_2x^2+a_1x}{x^2+x+b_0}$ for some $a_1,a_2,b_0\in \vF_q$ with $\Tr_{\vF_q/\vF_2}(b_0)=1$.
  \item $\varphi=\displaystyle \frac{x^3+a_2x^2+a_1x}{x^2+x+b_0}$ is permutation if and only if $\Tr_{\vF_q/\vF_2}(b_0)=1$, $\displaystyle a_1=b_0+\frac{1}{b_0}$ and $\displaystyle a_2=1+\frac{1}{b_0}$.
 \end{enumerate}
\end{theorem}
\begin{proof}
Let $f,g\in \vF_q[x]$ be coprime polynomials such that $\varphi=f/g$.

\begin{claim} $\varphi$ is equivalent to a rational function of the form $f'/g'$ where $f',g'\in\vF_q[x]$ are such that $\deg f'=3$ and $\deg g'=2$. 
\end{claim}
\begin{proof}[Proof of the claim] Up to composing with $1/x$ on the left, we can assume that $\deg f=3$ and $\deg g\leq 3$. 
 
If $\deg g=2$ there is nothing to prove. We cannot have $\deg g=1$ because the denominator cannot have any root $r$ as otherwise $\infty$ has two preimages: $r$ and $\infty$. If $\deg g=0$, then $\varphi$ is a polynomial and so it is equivalent to $x^3+ax$ for some $a\in \vF_q$. Notice that $a\neq 1$ because otherwise $\varphi(0)=\varphi(1)$. Composing on the right with $\frac{x}{x+1}$ and on the left with $1/x$ we get that $\varphi$ is equivalent to 
\[\varphi'=\frac{x^3+x^2+x+1}{(1+a)x^3+ax+1}.\] 
Now composing on the left with $\displaystyle \frac{x}{x+1/(1+a)}$ we get a function of the form $\displaystyle \frac{x^3+x^2+x+1}{x^2+bx+c}$ for some $b,c\in \vF_q$. 
 
Finally, if $\deg g=3$, up to rescaling the whole function we can assume that $g$ is monic. Now composing $\varphi$ on the left with $\displaystyle \frac{x}{x+1}$ yields a function of the form $f/g'$ with $\deg g'\leq 2$.
\end{proof}
 
 Let then $\displaystyle \frac{x^3+a_2'x^2+a_1'x+a_0'}{b_2'x^2+b_1'x+b_0'}$ be a permutation rational function, where $b_2'\neq 0$. Up to composing on the left by $b_2'x$, we can assume that $b_2'=1$. Now notice that $b_1'\neq 0$, since otherwise the denominator would have a root. Composing with $b_1'x$ on the right and composing on the left with $b_1'^2$, we get a function of the form $\displaystyle \frac{x^3+a_2''x^2+a_1''x+a_0''}{x^2+x+b_0''}$. Again, $b_0''\neq 0$, so composing on the left by $x+a_0''/b_0''$ and using Lemma \ref{quadratic_eq} we prove (1).
 
 In order to prove (2), we want to use part (3) of Lemma \ref{discriminant}, so let us assume that $q\geq 13$. The quadratic resolvent of $x^3+a_2x^2+a_1x-t(x^2+x+b_0)$ is given by:
 $$R_2(x)=x^2+(t^2+(a_1+a_2+b_0)t+a_1a_2)x+b_0t^4+(a_2b_0+1)t^3+(a_1+a_2^2b_0+b_0^2)t^2+(a_1^2+a_2^3b_0)t+a_1^3.$$
 In order for $R_2(x)$ to be irreducible over $\vF_q(t)$ but not over $\vF_{q^2}(t)$, it is necessary and sufficient that there exist $u_0,u_1,u_2\in \vF_{q^2}$, not all of which in $\vF_q$, such that
 $$R_2(x)=(x-(u_2t^2+u_1t+u_0))(x-(u_2^qt^2+u_1^qt+u_0^q)).$$
 This yields the following system of equations:
 
 \begin{equation}\label{spade}
\begin{cases}
    \Tr(u_0)=a_1a_2 & \\
    \Tr(u_1)=a_1+a_2+b_0 & \\
    \Tr(u_2)=1 & \\
    \N(u_0)=a_1^3 & \\
    \N(u_2)=b_0 & \\
    u_0u_1^q+u_0^qu_1=a_1^2+a_2^3b_0 & \\
    u_0u_2^q+\N(u_1)+u_0^qu_2=a_1+a_2^2b_0+b_0^2 & \\
    u_1^qu_2+u_1u_2^q=a_2b_0+1 & \\
   \end{cases},
\end{equation}
where all involved norms and traces are $\vF_{q^2}\to\vF_q$.
Let us now assume that $\varphi$ is a permutation and prove that the solutions to the equations \eqref{spade} satisfy $\Tr_{\vF_q/\vF_2}(b_0)=1$, $ a_1=b_0+1/b_0$ and $ a_2=1+1/b_0$.  
Let us start by assuming that $a_1a_2\neq 0$. Substituting the first three equations of \eqref{spade} into the last three, we can compute the norm of $u_1$ and we get two $\vF_q$-linear relations involving $u_0$ and $u_2$:
  $$\begin{cases}
     N(u_1)=\frac{a_1^5 + a_1^4a_2 + a_1^4 + a_1^3a_2b_0 + a_1^3b_0^2 + a_1^2a_2^4b_0 + a_1a_2^5b_0 + a_1a_2^4b_0^2 + a_2^6b_0^2}{(a_1a_2)^2} & \\
     (a_1+a_2+b_0)u_0+(a_1+a_2+b_0)a_1a_2u_2=a_1^2+a_2^3b_0+(a_2b_0+1)a_1a_2 & \\
     (a_1+a_2+b_0)u_0+(a_1+a_2+b_0)a_1a_2u_2=(a_1+a_2+b_0)(a_1+a_2^2b_0+b_0^2+N(u_1))& \\
    \end{cases},$$
   which yield a first relation among $a_1,a_2,b_0$:
   $$F:=(a_1^2 + a_1a_2 + a_1b_0 + a_2^2b_0)(a_1^4 + a_1^3a_2 + a_1^3 + a_1^2a_2^2b_0 + a_1^2a_2b_0 + a_1^2b_0^2 + a_1a_2^4b_0 + a_2^5b_0 + a_2^4b_0^2)=0$$
   Now the last equation of \eqref{spade} yields a linear relation between $u_1$ and $u_2$, namely $u_1=(a_1+a_2+b_0)u_2+a_2b_0+1$. Substituting this in the relation $u_1^2+\Tr(u_1)u_1+N(u_1)=0$ gives another relation among $a_1,a_2,b_0$:
   $$G:=(a_1 + a_2^2b_0)(a_1^4 + a_1^3a_2 + a_1^3 + a_1^2a_2^2b_0 + a_1^2a_2^2 + a_1^2a_2b_0 + a_1^2b_0^2 + a_1a_2^3 + a_1a_2^2b_0 + a_1a_2^2 + a_2^4b_0)=0.$$
   Now using Magma \cite{MR1484478}, we could find the prime components of the ideal $(F,G)\subseteq\vF_2[a_1,a_2,b_0]$. Let us list them below.
   $$\begin{cases}
    p_1=(a_1^2 + a_1 + b_0^2,a_2) & \\
    p_2=(a_1 + a_2 + b_0 + 1,a_2b_0 + b_0 + 1) & \\
    p_3=(a_1 + a_2 + b_0,a_2^2b_0 + a_2 + b_0) & \\
    p_4=(a_1 + a_2 + 1,b_0) & \\
    p_5=(a_1 + a_2,b_0) & \\
    p_6=(a_1 + b_0,a_2+1) & \\
    p_7=(a_1,a_2) & \\
    p_8=(a_1,b_0) & \\
   \end{cases},$$
By the assumption that $a_1a_2\neq 0$ (and the fact that $b_0\neq 0$ as we assumed that $\varphi$ to be a permutation)  we have that $a_1a_2b_0\neq 0$, therefore we can easily exclude $p_1,p_4,p_5,p_7,p_8$ from the list of necessary vanishing conditions. The vanishing of the generators of $p_6$ would imply $\varphi$ to have degree 1. If the generators of $p_3$ vanish, in particular the equation $b_0x^2+x+b_0$ has a root (equal to $a_2$) over $\vF_q$. By Lemma \ref{quadratic_eq}, this implies $\Tr_{\vF_q/\vF_2}(b_0^2)=0$, but of course $\Tr_{\vF_q/\vF_2}(b_0^2)=\Tr_{\vF_q/\vF_2}(b_0)$, and the latter is $1$ by hypothesis. Therefore, the generators of $p_2$ must vanish, which is equivalent to ask that \[\displaystyle a_1=b_0+\frac{1}{b_0}\quad\text{and}\quad  a_2=1+\frac{1}{b_0}.\] Notice that if $b_0=1$ then $a_1=a_2=0$.
    
Let us now deal with the case $a_1a_2=0$. Since the numerator of a permutation rational function cannot have two distinct roots, we must have $a_1=a_2=0$. Then from \eqref{spade} we get the equations
$$\begin{cases}
   u_1^2+b_0u_1+b_0^2=0 & \\
   u_2^2+u_2+b_0=0 & \\
   b_0u_2+u_1=1 & \\
  \end{cases}$$
  from which we get $b_0=1$. Since $x^2+x+1$ must be irreducible, $n$ is odd and therefore $\Tr_{\vF_q/\vF_2}(1)=1$.
    
    Conversely, assume that $b_0\in \vF_q$ is such that $\Tr_{\vF_q/\vF_2}(b_0)=1$ and $\displaystyle a_1=b_0+\frac{1}{b_0}$, $\displaystyle a_2=1+\frac{1}{b_0}$. Now let $u_2$ be a root of the equation $x^2+x+b_0$. By Lemma \ref{quadratic_eq}, we have that $u_2\in \vF_{q^2}\setminus \vF_q$. Letting \[ u_0\coloneqq \frac{(b_0+1)^3}{b_0^2}u_2 \quad\text{and}\quad u_1\coloneqq u_2+b_0,\] it is straightforward to check that equations \eqref{spade} hold.
    
  When $q<13$, Remark \ref{sufficient_condition} shows that the above conditions on $b_0,a_1,a_2$ are still sufficient for $\displaystyle \frac{x^3+a_2x^2+a_1x}{x^2+x+b_0}$ to be permutation. A direct check with Magma \cite{MR1484478} shows that they are also necessary. 
\end{proof}
\begin{corollary}\label{cor:equivalence_char_2}
 Any two permutation rational functions $\varphi,\psi$ of degree 3 in $\vF_q(x)$ are equivalent.
\end{corollary}
\begin{proof}
 Let $b_0,b_0'\in \vF_q$ be such that $\Tr_{\vF_q/\vF_2}(b_0)=\Tr_{\vF_q/\vF_2}(b_0')=1$. We claim that there exists $c\in \vF_q$ such that $b_0'=c^2+c+b_0$. Since the additive group of elements of trace zero has index two inside $\vF_q$, it is enough to show that every element of norm 0 can be written as $c^2+c$ for some $c\in \vF_q$. To do this, consider the $\vF_2$-linear map $y\mapsto y^2+y\colon \vF_q\to \vF_q$. The kernel of such map has $\vF_2$-dimension 1, so its image has dimension $n-1$. All elements in the image have trace 0, because $\Tr_{\vF_q/\vF_2}(z^2)=\Tr_{\vF_q/\vF_2}(z)$ for all $z\in \vF_q$. But there are exactly $2^{n-1}$ elements of trace 0, and the claim is proven.
 
 By Theorem \ref{char_2_class}, we can assume that 
 \[\displaystyle\varphi=\frac{x^3+a_2x^2+a_1x}{x^2+x+b_0}\quad \text{and}\quad  \psi=\frac{x^3+a_2'x^2+a_1'x}{x^2+x+b_0'}\] for some $a_1,a_2,b_0,a_1',a_2',b_0'\in \vF_q$. Notice that, again by point (2) of Theorem \ref{char_2_class}, if $b_0=b_0'$ then $\varphi=\psi$. Let $c\in \vF_q$ be such that $b_0'=c^2+c+b_0$. Now consider $\varphi(x+c)$. Up to a translation on the left, this has the form $\displaystyle\frac{x^3+ux^2+vx}{x^2+x+b_0'}$ for some $u,v\in \vF_q$. But then necessarily $u=a_2'$ and $v=a_1'$.
\end{proof}
\begin{corollary}
 Let $N_q$ be the number of rational functions of degree 3 of the form $f/g$, where $f,g\in \vF_q[x]$ are monic coprime polynomials. Let $m\in \{1,2\}$ be the residue class of $q$ modulo $3$. Then:
  $$N_q=\frac{1}{2}(q^4+2(m-1)q^3+(2m-3)q^2).$$
\end{corollary}
\begin{proof}
 The proof is essentially the same as that of Corollary \ref{number_of_permutations}. Let $R_{s,t}$ be the set of permutation rational functions of degree 3 of the form $f/g$, where $f,g$ are monic polynomials of degree $s,t$, respectively. Then $N_q=|R_{3,3}|+2|R_{3,2}|+2|R_{3,0}|$. One sees immediately that $|R_{3,0}|=(m-1)q^2$.
 Let $R_{3,2}'\subseteq R_{3,2}$ be the subset of permutation rational functions of the form $\displaystyle \frac{x^3+a_2x^2+a_1x+a_0}{x^2+x+b_0}$. By composing on the left with $x+a_0/b_0$ and using Theorem \ref{char_2_class} one sees that such function is permutation if and only if $\Tr_{\vF_q/\vF_2}(b_0)=1$, $a_1=b_0+1/b_0+a_0/b_0$ and $a_2=1+1/b_0+a_0/b_0$. Thus, the choices of $a_0,b_0$ determine the function completely, and since there are $q/2$ elements of norm 1, one gets $|R_{3,2}'|=q^2/2$. Now let \[ \varphi\coloneqq\frac{x^3+a_2x^2+a_1x+a_0}{x^2+b_1x+b_0}\in R_{3,2}.\] The map $R_{3,2}\to R_{3,2}'$ defined by $\varphi\mapsto \frac{1}{b_1}\varphi(b_1x)$ is surjective with fibers of cardinality $(q-1)$.
 
 The cardinality of $R_{3,3}$ is computed in the very same way than in Corollary \ref{number_of_permutations}, and the result is the same.
\end{proof}

\section{Complete rational functions of degree three}\label{sec:ComRatDeg3}

A \emph{complete permutation polynomial} is a permutation polynomial $f\in\vF_q[x]$ such that $f+x$ is a permutation polynomial as well. See \cite{MR3622120,MR3528707} for results towards the classification of complete permutation polynomials. Construction of complete permutation polynomials are of great interest for applications (for example in the context of mutually ortogonal latin squares).
 In this section, we will show how our results naturally allow to prove that there are no complete permutation rational functions of degree three, the only exception being complete permutation polynomials over fields of characteristic 3. Corollary \ref{degree_3_class} shows that these are exactly polynomials of the form $ax^3+bx+c\in \vF_{3^n}[x]$, where $a\neq 0$ and both $-b/a$ and $-(b+1)/a$ are either 0 or non-squares in $\vF_{3^n}$.

\begin{theorem}
 Let $\varphi\in \vF_q(x)$ be a rational function of degree three. Then $\varphi$ is complete if and only if $3\mid q$ and $\varphi$ is a complete permutation polynomial.
\end{theorem}

\begin{proof}

 One direction is obvious, so let us prove the converse. First, assume that $\varphi$ is a complete permutation polynomial. Notice that if $\lambda,\mu,\nu\in\vF_q$ and $\lambda\neq 0$, then $\varphi'=\lambda\varphi(x+\mu)+\nu$ is a permutation polynomial such that $\varphi'+\lambda x$ is permutation. We shall call such a rational map ``$\lambda$-complete'', for shortness. Clearly the appropriate choice of the triple $(\lambda,\mu,\nu)$ transforms $\varphi$ into a normalized $\lambda$-complete permutation polynomial, for some $\lambda\in\vF_q^*$. Corollary \ref{degree_3_class} immediately implies that $3\mid q$.
 
 Now assume that $\varphi=f/g$ is a complete permutation rational function, where $f,g\in \vF_q[x]$ are coprime and $\deg g>0$. Suppose first that $\deg f\leq \deg g$. Since $\varphi$ is permutation, there must exist a unique $t\in\vF_q$ such that $\varphi(t)=\infty$. On the other hand, $\varphi'\coloneqq\varphi+x$ is also permutation, and $\varphi'(t)=\varphi(\infty)=\infty$, which is a contradiction. It follows that $\deg f>\deg g$. Thus, $\deg f=3$ and $\deg g=2$ (one cannot have $\deg g=1$ because $\infty$ would have two pre-images via $\varphi$). As noticed above, $\lambda\varphi(x+\mu)+\nu$ is $\lambda$-complete for every $\lambda,\mu,\nu\in \vF_q$ with $\lambda\neq 0$. Therefore, choosing the triple $(\lambda,\mu,\nu)$ appropriately, we can assume that $\varphi$ is $\lambda$-complete and, if $q$ is odd, that $f=x^3+a_2x^2+a_1x$ and $g=x^2+b_0$ for some $b_0,a_1,a_2\in \vF_q$ with $-b_0$ not a square, while if $q$ is even that $f=x^3+a_2x^2+a_1x$ and $g=x^2+x+b_0$ for some $b_0,a_1,a_2\in\vF_q$ with $g$ irreducible. In the first case, we can use Theorem \ref{complete_classification} as the denominator has the degree $1$ coefficient equal zero. This forces that $a_2=0$ and $a_1/b_0=9$. Now $\varphi+\lambda x=\frac{(1+\lambda)x^3+(a_1+b_0\lambda)x}{x^2+b_0}$, and we cannot have $\lambda=-1$ because there are no permutation rational functions of degree two in odd characteristic. The same theorem shows therefore that $\frac{a_1+b_0\lambda}{1+\lambda}=a_1$, which in turn implies that $a_1=b_0$, a contradiction because $f$ and $g$ are coprime by hypothesis. If $2\mid q$, one computes $\varphi+\lambda x=\frac{(1+\lambda)x^3+(\lambda+a_2)x^2+(a_1+b_0\lambda)x}{x^2+x+b_0}$. Notice that if $\lambda=1$, then $\varphi+\lambda x$ cannot be permutation because $\infty$ would have no pre-images via $\varphi+\lambda x$. Thus $\lambda\neq 1$ and Theorem \ref{char_2_class} implies that $\frac{\lambda+a_2}{1+\lambda}=a_2=1+1/b_0$, yielding a contradiction.
\end{proof}

\section{Final remarks and further research}

Combining the results in Sections \ref{sec:degree_three}, \ref{sec:odd_char}, and  \ref{sec:even_char} we get the results in Table \ref{table_permutation_up_to_equivalence} and Table \ref{tnumber_of_permutations}. It is worth noticing that as a side result we have the following theorem.
\begin{theorem}
A degree $3$ rational function over $\vF_q$ permutes $\vP^1(\vF_q)$ if and only if it permutes infinitely many extensions of $\vF_q$.
\end{theorem}
This is due to the fact that below the regime in which Theorem \ref{fixed_points} holds, there is no ``degenerate'' case, where the pair (arithmetic Galois group, geometric Galois group) differs from $(S_3,A_3)$. It would be very interesting to see what happens in higher degrees: for example, we already know that ther are no permutation polynomials of degree $4$ when the size of the field is large \cite[Remark 8.1.14]{mullen2013handbook} and the characteristic is odd, whilst there are permutation rational functions (see \cite{guralnick2003rational}). On the other hand, we also know that in higher degrees there are degenerate cases (e.g. $x^4+3x\in \vF_7[x]$).
Using Theorem \ref{fixed_points} one can also characterize permutation rational functions of degree 4 in terms of the factorization of the cubic resolvent. 
For the sake of completeness we include the statement of the result here.

\begin{proposition}\label{degree_4}
Let $q$ be a prime power such that $\gcd(q,6)=1$ and $\varphi=f/g\in \vF_q(x)$ be a separable rational function of degree 4, where $f,g$ are coprime polynomials.
Let $R_3(x)$ be the cubic resolvent of the polynomial $f-gt\in \vF_q(t)[x]$. Suppose that $q>C(4)$, where $C(4)$ is an absolute constant.
The following three conditions are equivalent.
 \begin{enumerate}
  \item The function $\varphi$ is permutation.
  \item The arithmetic Galois group $G$ is the alternating group $A_4$ and the geometric Galois group $N$ is the Klein four-group $V_4$.
  \item There exist $u,v \in \vF_{q^3}$ such that at least one between $u,v$ does not lie in $\vF_q$ and the following factorization holds:
  $$R_3(x)=(x-(ut+v))(x-(u^qt+v^q))(x-(u^{q^2}t+v^{q^2})).$$
 \end{enumerate}
\end{proposition}

\begin{problem}
Compute the number of equivalence classes of degree $4$ permutation rational functions and give explicitely a representative for each class.
\end{problem}
An approach is to prove the analogous of point (3) of the above characterization for $\gcd(q,6)\neq 1$. Then, deal separately with the degenerate cases for $q$ below the absolute constant $C(4)$.
This should in turn lead to solve the following
\begin{problem}
Give exact formulas (depending only on $q$) for the number of permutation rational functions of degree $4$.
\end{problem}

\bibliographystyle{plain}
\bibliography{biblio}

\end{document}